\documentclass[12pt]{amsart}
\usepackage[english,activeacute]{babel}
\usepackage{amsmath,amsfonts,amssymb,amstext,amsthm,amscd,mathrsfs,amsbsy}
\usepackage{graphics,graphicx}

\newtheorem{teo}{Theorem}[section]
\newtheorem{pro}[teo]{Proposition}
\newtheorem{coro}[teo]{Corollary}
\newtheorem{lem}[teo]{Lemma}

\theoremstyle{definition}
\newtheorem{defi}[teo]{Definition}
\newtheorem{exam}[teo]{Example}
\newtheorem{rem}[teo]{Remark}
\newtheorem{nota}[teo]{Notation}

\newcommand{\V}{\mathbf V}
\newcommand{\I}{\mathbf I}

\newcommand{\R}{\mathbb R}
\newcommand{\N}{\mathbb N}
\newcommand{\Z}{\mathbb Z}
\newcommand{\C}{\mathbb C}

\newcommand{\Oo}{\mathcal O}

\newcommand{\sd}{\check{\sigma}}
\newcommand{\Si}{\mbox{Sing}(X)}
\newcommand{\rk}{\mbox{rank}}

\newcommand{\J}{\mathcal J}
\newcommand{\Jc}{\mbox{Jac}}
 
\newcommand{\VJ}{\mathbf V(\mathcal J)}
\newcommand{\IG}{I_{\Gamma}}
\newcommand{\cero}{\mathbf{0}}
\newcommand{\zl}{\mathbf{0}_l}
\newcommand{\zm}{\mathbf{0}_m}
\newcommand{\zn}{\mathbf{0}_n}

\newcommand{\ul}{\mathbf{1}_l}
\newcommand{\um}{\mathbf{1}_m}
\newcommand{\un}{\mathbf{1}_n}

\begin{document}

\title{On the zero locus of ideals defining the Nash blowup of toric surfaces}

\author{Enrique Ch\'avez Mart\'inez, Daniel Duarte}

\maketitle

\begin{abstract}
We consider the problem of finding an ideal whose blowup defines the Nash blowup of a toric surface and such 
that its zero locus coincides with the singular set of the toric surface.
\end{abstract}


\section*{Introduction}

The Nash blowup of an equidimensional algebraic variety $X$ is a modification that replaces singular points 
by limits of tangent spaces to non-singular points. It has been proposed to iterate the Nash blowup to solve 
singularities of algebraic varieties (see \cite{No, S}). This question has been studied in 
(\cite{No,R,GS-1,GS-2,Hi,Sp,At,GM,GT,D}).

In \cite{No, GS-1} a method to compute an ideal whose blowup defines the Nash blowup of $X$ is provided. 
It is also shown that if $\J$ is such an ideal then $\Si\subset\V(\J)$, where $\Si$ denotes the singular locus of $X$. 
In addition, if $X$ is a complete intersection, then $\Si=\V(\J)$. The question we are interested in is the following: 
\textit{For non-complete intersections, is there a choice of an ideal defining the Nash blowup such that its zero locus 
coincides with the singular locus of $X$?}

This question is a simplified version of a problem that appears in \cite{Gi}. In that paper, the zero set of an ideal defining 
the \textit{relative} Nash blowup of certain analytic spaces is studied. The author shows that being able to have some control 
over that set implies important results concerning equisingularity problems (see \cite[Section 8]{Gi}). Our study was motivated 
by that consideration. We present here a first step in the study of the zero locus of different ideals defining the 
Nash blowup of toric varieties.

As is usually the case with toric varieties, the question can be explored using combinatorial methods. We explore the case of
toric varieties and specifically of dimension two mainly for three reasons:
\begin{enumerate}
\item Minors of the Jacobian matrix of the defining equations of a toric variety are monomials (cf. lemma \ref{l. minors}).
\item Existence of an explicit set of defining equations of a toric surface (cf. corollary \ref{c. equations}).
\item A characterization of complete intersection for normal toric surfaces (see \cite{Ri}). 
\end{enumerate}
Using these facts we will show that the answer to the question is positive when the dimension of the singular locus is 1. 
We will also show that the answer is negative if the dimension of the singular locus is 0, whenever the ideal defining the 
Nash blowup is obtained from the binomial ideal of the toric surface and by using the method appearing in \cite{No, GS-1}.

We want to emphasize that our result gives a positive consequence of removing the hypothesis of normality in the definition 
of toric varieties. It is known that some nice properties of normal toric varieties are no longer valid in the non-normal case. 
For instance, normal toric varieties are Cohen-Macaulay (\cite{Dan}) and a normal toric surface is a complete intersection if and
only if it is a hypersurface (\cite{Ri}). Both results are no longer true without the assumption of normality 
(see \cite[Examples 6.5.6(6)]{dJP} and \cite[Section 4.1.2]{Ch}). 
Our result illustrates a nice property that holds only for some non-normal toric surfaces.


\section{The Nash blowup of a toric variety}

Let us start by recalling the definition of the Nash blowup of an equidimensional algebraic variety.

\begin{defi}
Let $X\subset\C^s$ be an equidimensional algebraic variety of dimension $d$. Consider the Gauss map:
\begin{align}
G:X\setminus\Si&\rightarrow Gr(d,s)\notag\\
x&\mapsto T_xX,\notag
\end{align}
where $Gr(d,s)$ is the Grassmanian of $d$-dimensional vector spaces in $\C^s$, and $T_xX$ is the
tangent space to $X$ at $x$. Denote by $X^*$ the Zariski closure of the graph of $G$. Call $\nu$ 
the restriction to $X^*$ of the projection of $X\times Gr(d,s)$ to $X$. The pair $(X^*,\nu)$ is called the Nash 
blowup of $X$.
\end{defi}

By construction, $\nu$ is a birational morphism (it is an isomorphism over the set of non-singular points). 
The following theorem gives a method to find an ideal whose blowup defines the Nash blowup 
(see \cite[Theorem 1, Remark 2]{No}, \cite[Chapitre 2, Section 1]{GS-1}).

\begin{teo}\label{t. ideal Nash}
Let $X\subset\C^s$ be an irreducible algebraic variety of dimension $d$ and $\C[X]$ its coordinate ring. Let $r=s-d$. 
Let $\{f_1,\ldots,f_r\}\subset\I(X)$ be any set of $r$ polynomials such that the Jacobian matrix $M=\Jc(f_1,\ldots,f_r)$ has rank $r$. 
Let $\J\subset\C[X]$ be the ideal generated by all $r\times r$ minors of $M$. Then the blowup of $X$ along $\J$ is isomorphic 
to the Nash blowup of $X$. In addition, if $X$ is a complete intersection then $f_1,\ldots,f_r$ can be taken as a minimal set of 
generators of $\I(X)$.
\end{teo}

\begin{rem}\label{r. inclusion}
Notice that theorem \ref{t. ideal Nash} implies $\Si\subseteq\V(\J)$. 
\end{rem}

We are interested in studying the ideal $\J$ of the previous theorem in the case of toric surfaces. Let us recall the
definition of an affine toric variety (see, for instance, \cite[Section 1.1]{CLS} or \cite[Chapters 4, 13]{St}).

Let $\Gamma\subset\Z^d$ be a semigroup generated by a finite set $\Gamma_0=\{\gamma_1,\ldots,\gamma_s\}\subset\Z^d$  
such that $\Z\Gamma_0=\{\sum_i \lambda_i\gamma_i|\lambda_i\in\Z\}=\Z^d$. The set $\Gamma_0$ 
induces a homomorphism of semigroups
\begin{align}\label{e. pi}
\pi_{\Gamma_0}:\N^s\rightarrow\Z^d,\mbox{ }\mbox{ }\mbox{ }\alpha=(\alpha_1,\ldots,\alpha_s)
\mapsto \alpha_1\gamma_1+\cdots+\alpha_s\gamma_s.
\end{align}
Consider the ideal
$$I_{\Gamma}:=\langle x^{\alpha}-x^{\beta}|\alpha,\beta\in\N^s,\mbox{ }\pi_{\Gamma_0}(\alpha)=\pi_{\Gamma_0}
(\beta)\rangle\subset\C[x_1,\ldots,x_s].$$

\begin{defi}
We call $X:=\V(I_{\Gamma})\subset\C^s$ the toric variety defined by $\Gamma$.
\end{defi}


It is well known that a variety obtained in this way is irreducible, contains a dense open set isomorphic 
to $(\C^*)^d$ and such that the natural action of $(\C^*)^d$ on itself extends to an action on $X$. 

\begin{rem}
In the case of toric varieties, there is a combinatorial way to compute an ideal whose blowup defines the Nash blowup.
This fact is a consequence of theorem \ref{t. ideal Nash} (see \cite{GS-1,GT,LJ-R}). 
\end{rem}

According to remark \ref{r. inclusion}, $\Si\subseteq\V(\J)$. In addition, theorem \ref{t. ideal Nash} states that,
for complete intersections, $\Si=\V(\J)$. The question we are interested in is the following: For toric surfaces $X$ which are 
non-complete intersections, is there a choice of an ideal $\J$ defining the Nash blowup of $X$ such that $\Si=\V(\J)$?
Our main theorem gives a partial answer to this question. 


\subsection{An example}

Let us start our discussion with an example. All examples that appear in this note were computed using
the software $\mathtt{SINGULAR}$ $\mathtt{4}$-$\mathtt{0}$-$\mathtt{2}$ (\cite{DGPS}).

Let $\Gamma\subset\Z^2$ be the semigroup generated by $\{(1,0),(1,1),(1,2),(1,3)\}$. Using 
\cite[Algorithm 4.5]{St}, we can compute the ideal associated to $\Gamma$ which is 
$$I_{\Gamma}=\langle x_1x_3-x_2^2, x_1x_4-x_2x_3, x_2x_4-x_3^2\rangle.$$
Let $X=\V(\IG)\subset\C^4$ be the respective toric surface. It can be checked that
$(0,0,0,0)\in X$ is the only singular point (actually, $X$ is a normal surface).  
Consider the Jacobian matrix of $\IG$:
\[
\Jc(\IG)=
\begin{pmatrix}
x_3& 	-2x_2& 	x_1&  	 0\\
x_4&  	 -x_3&  -x_2&   x_1\\
0&       x_4&  -2x_3&    x_2
\end{pmatrix}
\]
Denote by $\J_{ij}$ the ideal generated by $2\times 2$ minors of the matrix formed by the rows 
$i,j\in\{1,2,3\}$, $i\neq j$, of $\Jc(\IG)\mod\IG$.
We have
\begin{align}
\J_{12}&=\langle x_1^2,x_1x_2,x_1x_3,x_1x_4,x_2x_4\rangle,\notag\\
\J_{13}&=\langle x_1x_2,x_1x_3,x_1x_4,x_2x_4,x_3x_4\rangle,\notag\\
\J_{23}&=\langle x_1x_3,x_1x_4,x_2x_4,x_3x_4,x_4^2\rangle.\notag
\end{align}
According to theorem \ref{t. ideal Nash}, these ideals define the Nash blowup of $X$. 
Notice that the zero locus of these ideals strictly contain $\Si$.


\section{Defining equations and orbits of a toric surface}\label{s. gamma}

In this section we prepare the notation and basic properties regarding defining equations and orbits of
a toric surface. Let 
$$\Gamma_0=\{(a_1,b_1),\ldots,(a_l,b_l),(c_1,d_1),\ldots,(c_m,d_m),(e_1,f_1),\ldots,(e_n,f_n)\}\subset\Z^2.$$
Let $\Gamma:=\Z_{\geq0}\Gamma_0\subset\Z^2$ and $\sd:=\R_{\geq0}\Gamma_0\subset\R^2$ be the semigroup 
and the cone generated by $\Gamma_0$, respectively. Denote by $\sigma\subset\R^2$ the dual cone of $\sd$. 
We assume that (see figure \ref{f. gamma}):
\begin{itemize}
\item[(i)] $l\geq1$, $m\geq0$, $n\geq1$, and $\Gamma_0$ is a minimal set of generators of $\Gamma$.
\item[(ii)] $\Z\Gamma=\Z^2$. In particular, $\dim\sd=2$ and $\sigma$ is strictly convex.
\item[(iii)] $\sd$ is strictly convex.
\item[(iv)] The set $\{(a_i,b_i)|i=1,\ldots,l\}$ generate one of the edges of $\sd$.
\item[(v)] The set $\{(e_k,f_k)|k=1,\ldots,n\}$ generate the other edge of $\sd$.
\item[(vi)] The set $\{(c_j,d_j)|j=1,\ldots,m\}$ is contained in the relative interior of $\sd$.
\item[(vii)] $\cero\in X\subset\C^{l+m+n}$ is a singular point, where $X$ is the toric surface defined by $\Gamma$.
\end{itemize}
\textit{Every toric surface that we will consider from now on is defined by a semigroup satisfying these conditions. 
In particular, from now on $\Gamma$ will denote a semigroup satisfying these conditions.}

\begin{figure}[ht]
\begin{center}
\includegraphics{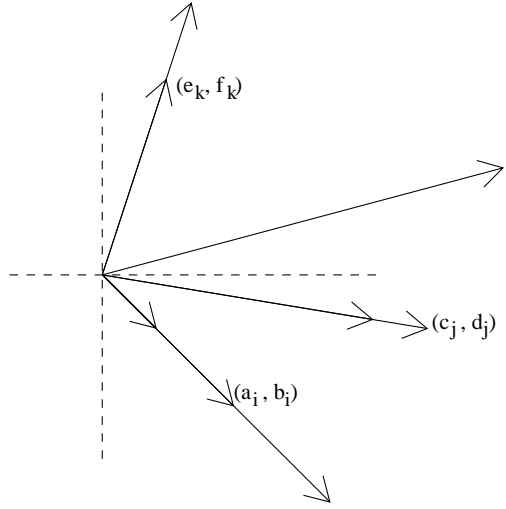}
\caption{The set of generators of $\Gamma$.\label{f. gamma}}
\end{center}
\end{figure}

\begin{nota}
We will use the following notation:
\begin{itemize}
\item The coordinates of $\C^{l+m+n}$ are named $x_1,\ldots,x_l$, $y_1,\ldots,y_m$, $z_1,\ldots,z_n$.
\item $\zl:=(0,\ldots,0)\in\N^{l}$, $\ul:=(1,\ldots,1)\in\N^{l}$. Similarly, we define $\zm,\zn,\um,\un$.
In addition, $\cero:=(\zl,\zm,\zn)$.
\end{itemize}
\end{nota}

\subsection{Defining equations of a toric surface}

In the case of normal toric surfaces, a set of minimal generators for its defining ideal can be computed
explicitly (see \cite{Ri}). Unfortunately, as far as we know, there is no analogous result for the non-normal case.
In this section we determine some of the relations among the elements of $\Gamma_0$.

\begin{lem}
\begin{itemize}
\item[(1)] For any $i,j\in\{1,\ldots,l\}$, $i\neq j$, there exist $\alpha_{ij},\alpha'_{ij}\in\N$ such that 
$\alpha_{ij}(a_i,b_i)=\alpha'_{ij}(a_j,b_j)$.
\item[(2)] For any $i,j\in\{1,\ldots,n\}$, $i\neq j$, there exist $\gamma_{ij},\gamma'_{ij}\in\N$ such that 
$\gamma_{ij}(e_i,f_i)=\gamma'_{ij}(e_j,f_j)$.
\item[(3)] For any $i\in\{1,\ldots,m\}$, there exists $\alpha_{i},\beta_{i},\gamma_{i}\in\N$ such that
$\beta_{i}(c_i,d_i)=\alpha_{i}(a_1,b_1)+\gamma_{i}(e_1,f_1)$.
\end{itemize}
\end{lem}
\begin{proof}
The set of points with integer entries on a line of rational slope in $\R^2$ is a lattice of rank 1, so it is 
isomorphic to $\Z$. Since the cone $\sd$ is strictly convex, statements $(1)$ and $(2)$ follow.
To prove $(3)$ we need to find a solution $(\alpha,\beta,\gamma)\in\N^3$ of the following system:
\begin{align}
\beta c_i=\alpha a_1+\gamma e_1,\notag\\
\beta d_i=\alpha b_1+\gamma f_1,\notag
\end{align}
where $a_1f_1-b_1e_1\neq0$, $a_1d_i-b_1c_i\neq0$, $c_if_1-d_ie_1\neq0$ and $(c_i,d_i)$ is contained in the 
relative interior of the cone $\R_{\geq0}((a_1,b_1),(e_1,f_1))$. An elementary computation shows that under
these conditions the system do have a solution in $\N^3$.
\end{proof}

\begin{coro}\label{c. equations}
\begin{itemize}
\item[(1)] For any $i,j\in\{1,\ldots,l\}$, $i\neq j$, there exist $\alpha_{ij},\alpha'_{ij}\in\N$ such that 
$x_i^{\alpha_{ij}}-x_j^{\alpha'_{ij}}\in I_{\Gamma}$.
\item[(2)] For any $i,j\in\{1,\ldots,n\}$, $i\neq j$, there exist $\gamma_{ij},\gamma'_{ij}\in\N$ such that 
$z_i^{\gamma_{ij}}-z_j^{\gamma'_{ij}}\in I_{\Gamma}$.
\item[(3)] For any $i\in\{1,\ldots,m\}$, there exists $\alpha_{i},\beta_{i},\gamma_{i}\in\N$ such that
$y_i^{\beta_{i}}-x_1^{\alpha_{i}}z_1^{\gamma_i}\in I_{\Gamma}$.
\end{itemize}
\end{coro}

\begin{lem}\label{l. x-x}
Suppose that there exist $\alpha_1,\alpha_2\in\N^l$, $\beta\in\N^m$, $\gamma\in\N^n$ such that 
$x^{\alpha_1}-x^{\alpha_2}y^{\beta}z^{\gamma}\in I_{\Gamma}$. Then
$\beta=\zm$ and $\gamma=\zn$. Similarly, if $z^{\gamma_1}-x^{\alpha}y^{\beta}z^{\gamma_2}\in I_{\Gamma}$
then $\alpha=\zl$ and $\beta=\zm$.
\end{lem}
\begin{proof}
It is known that $x^{\alpha_1}-x^{\alpha_2}y^{\beta}z^{\gamma}\in I_{\Gamma}$ if and
only if $\pi_{\Gamma_0}(\alpha_1,\zm,\zn)=\pi_{\Gamma_0}(\alpha_2,\beta,\gamma)$, where $\pi_{\Gamma_0}$
is the homomorphism defined in (\ref{e. pi}). Since the subset $\{(a_i,b_i)|i=1,\ldots,l\}\subset\Gamma_0$ generates one 
of the edges of $\sd$, and by hypothesis 
$$\sum_{i=1}^{l} \alpha_{1i}(a_i,b_i)=\sum_{i=1}^{l}\alpha_{2i}(a_i,b_i)+\sum_{j=1}^{m}\beta_{j}(c_j,d_j)+\sum_{k=1}^{n}\gamma_{k}(e_k,f_k),$$
it follows that $\beta=\zm$ and $\gamma=\zn$. The other statement follows similarly.
\end{proof}

\begin{coro}\label{c. z-z}
Let $X\subset\C^{l+m+n}$ be the toric surface defined by $\Gamma$. Then $(\ul,\zm,\zn)\in X$
and $(\zl,\zm,\un)\in X$.
\end{coro}
\begin{proof}
Since $\cero\in X$, for every binomial $x^{\alpha}y^{\beta}z^{\gamma}-x^{\alpha'}y^{\beta'}z^{\gamma'}\in\IG$,
we must have $(\alpha,\beta,\gamma)\neq(\zl,\zm,\zn)$ and $(\alpha',\beta',\gamma')\neq(\zl,\zm,\zn)$.
On the other hand, $\IG$ is generated by binomials by definition. The corollary now follows from the lemma.
\end{proof}


\subsection{Orbits on a toric surface}

Let us call $\tau_1$ and $\tau_2$ the rays of $\sigma$ (see figure \ref{f. sigma dual}). 

\begin{figure}[ht]
\begin{center}
\includegraphics{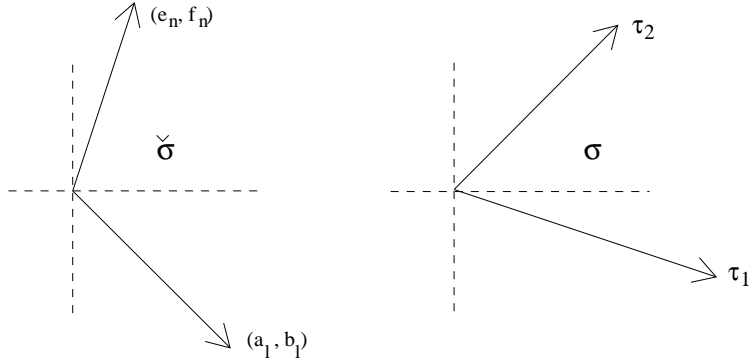}
\caption{The cone $\sd$ and its dual $\sigma$.\label{f. sigma dual}}
\end{center}
\end{figure}

The torus $T\subset X$ is given by $X\cap(\C^*)^{l+m+n}$. In addition, the action of $T$ on $X$ is given by 
coordinatewise multiplication (see \cite[Lemma 13.4]{St}). On the other hand, it is known that there is a bijective correspondence 
between the orbits of the action on $X$ and the faces of $\sigma$ (see \cite[Section 3.2]{CLS} for the normal case, 
or \cite[Proposition 15]{GT} for the general case). In particular, there are four such orbits. Using corollary \ref{c. z-z}, we
can describe them as follows:
\begin{align}\label{e. orbits}
\Oo_0\simeq & (\C^*)^2  \mbox{ (the orbit corresponding to the face }\{0\}),\notag\\
\Oo_{\sigma}=&\{\cero\} \mbox{ (the orbit corresponding to the face }\sigma),\notag\\
\Oo_1=&T\cdot(\zl,\zm,\un)=\{(\zl,\zm,z_1,\ldots,z_n)\in X|z_k\neq0\mbox{ for all }k\} \notag\\
&\mbox{ (the orbit corresponding to the face }\tau_1),\notag\\
\Oo_2=&T\cdot(\ul,\zm,\zn)=\{(x_1,\ldots,x_l,\zm,\zn)\in X|x_i\neq0\mbox{ for all }i\}\notag\\
&\mbox{ (the orbit corresponding to the face }\tau_2).\notag
\end{align}

\begin{lem}\label{l. monomial orbits}
Let $X\subset\C^{l+m+n}$ be the toric surface defined by $\Gamma$, $J\subsetneq\C[X]$ a non-zero monomial ideal. 
Then $\V(J)\subsetneq X$ is a union of closures of orbits of the action.
\end{lem}
\begin{proof}
This is clear since the zero set of a monomial ideal is invariant under the action of the torus.
\end{proof}

\section{The zero locus of $\J$}

Let us start by preparing the notation that we will use throughout this section.

Let $\Gamma\subset\Z^2$ be as in the beginning of section \ref{s. gamma} and let $X\subset\C^{l+m+n}$ be the toric surface
defined by $\Gamma$. Consider a set of relations among the generators of $\Gamma$:
\begin{align}
r_{\lambda}: \sum_{i=1}^l A_{\lambda i}(a_i,b_i)&+\sum_{j=1}^m B_{\lambda j}(c_j,d_j)+\sum_{k=1}^n C_{\lambda k}(e_k,f_k)=\notag\\
&\sum_{i=1}^l A'_{\lambda i}(a_i,b_i)+\sum_{j=1}^m B'_{\lambda j}(c_j,d_j)+\sum_{k=1}^n C'_{\lambda k}(e_k,f_k),\notag
\end{align}
where $\lambda=1,\ldots,s$, and $A_{\lambda i},B_{\lambda j},C_{\lambda k}\in\N$. Denote by $A_{\lambda}=(A_{\lambda 1},\ldots,A_{\lambda l})$ 
and similarly for $B_{\lambda}$, $C_{\lambda}$, $A'_{\lambda}$, $B'_{\lambda}$, $C'_{\lambda}$.
We assume that $\{r_{\lambda}\}$ is a generating family of relations.


\begin{rem}\label{r. minimal}
Notice that no $(A_{\lambda},B_{\lambda},C_{\lambda})$ is of the form $(0,\ldots,0,1,0,\ldots,0)$ and similarly for 
$(A'_{\lambda},B'_{\lambda},C'_{\lambda})$. This is true because $\Gamma_0$ is a minimal set of generators of $\Gamma$.
\end{rem}

\begin{rem}\label{r. rank}
Let $f_{\lambda}=x^{A_{\lambda}}y^{B_{\lambda}}z^{C_{\lambda}}-x^{A'_{\lambda}}y^{B'_{\lambda}}z^{C'_{\lambda}}$, 
$\lambda=1,\ldots,s$. Since $\{r_{\lambda}\}$ is a generating family of relations, $I_{\Gamma}=\langle f_1,\ldots,f_s \rangle\subset\C[x,y,z]$. 
Let $r=l+m+n-2$. After renaming binomials if necessary, 
we may assume that $f_1,\ldots,f_r$ are such that $\rk(\Jc(f_1,\ldots,f_r))=r$.
\end{rem}

\begin{nota}
From now on, $\J\subset\C[x,y,z]/\IG$ denotes the ideal generated by all $r\times r$ minors of $\Jc(f_1,\ldots,f_r)$.
\end{nota}

Using the notation of remark \ref{r. rank}, we can now look for an ideal defining the Nash blowup of a toric surface 
using theorem \ref{t. ideal Nash}. To that end we need to compute certain minors of the Jacobian matrix of
$f_1,\ldots,f_r$. It turns out that those minors have a combinatorial description which is explained in the following lemma
(in \cite[Section 2, Lemme 2]{GS-1} and \cite[Proposition 60]{GT}, the lemma is proved for affine toric varieties of 
any dimension).

\begin{lem}\label{l. minors}
Let $K_1\subset\{1,\ldots,l\}$, $K_2\subset\{1,\ldots,m\}$, $K_3\subset\{1,\ldots,n\}$ and $K=K_1\sqcup K_2\sqcup K_3$.
Assume that $|K|=2$. Let $M$ be the $(r\times r)$-submatrix of 
$\Jc(f_1,\ldots,f_r)$ formed by all its rows and all its columns except the two columns corresponding to the variables
associated to $K$. Then $\det(M)\mod\IG$ equals 
$$\det(R_K)\cdot \frac{x^{A_1+\cdots+A_r}y^{B_1+\cdots+B_r}z^{C_1+\cdots+C_r}}
{x_1\cdots x_l y_1\cdots y_m z_1\cdots z_n}\cdot\prod_{i\in K_1}x_i\prod_{j\in K_2}y_j\prod_{k\in K_3}z_k,$$
where $R_K$ is the $(r\times r)$-submatrix of $(A_{\lambda i}-A'_{\lambda i}\mbox{ }B_{\lambda j}-B'_{\lambda j}
\mbox{ }C_{\lambda k}-C'_{\lambda k}),$ $\lambda=1,\ldots,r$, $i=1,\ldots,l$, $j=1,\ldots,m$, $k=1,\ldots,n$, formed by 
all its rows and all its columns except the two corresponding to the variables associated to $K$.
In particular, $\J\subset\C[x,y,z]/\IG$ is a monomial ideal.
\end{lem}


\begin{coro}\label{c. union orbits}
Let $X$ be the toric surface defined by $\Gamma$. 
Then $\V(\J)\subsetneq X$ is a union of closures of orbits of the action. Similarly, $\Si$ is also a union
of closures of orbits of the action.
\end{coro}
\begin{proof}
By lemma \ref{l. minors}, $\J$ is a monomial ideal. The statement on $\V(\J)$ follows from lemma \ref{l. monomial orbits}.
On the other hand, by repeteadly applying the lemma for all choices of $r$ binomials among $f_1,\ldots,f_s$,
we obtain that the ideal generated by all $r\times r$ minors of $\Jc(f_1,\ldots,f_s)$ is also monomial. 
Thus $\Si$ is a union of closures of orbits by lemma \ref{l. monomial orbits}.
\end{proof}


We are looking for a choice of binomials in $\IG$ 
satisfying the conditions of theorem \ref{t. ideal Nash} and such that $\Si=\V(\J)$. We consider this problem 
according to the dimension of the singular locus. In the next two sections we prove our main theorem.

\begin{teo}\label{t.  main}
Let $X\subset\C^{l+m+n}$ be the toric surface defined by $\Gamma$ and assume that $X$ is not a complete intersection.
\begin{itemize}
\item[(i)] If $\dim\Si=1$, there is a choice of $r$ binomials among $f_1,\ldots,f_s$ such that
the corresponding ideal $\J$ satisfies $\Si=\V(\J)$.
\item[(ii)] If $\dim\Si=0$, there are no $r$ binomials among $f_1,\ldots,f_s$ such that 
the corresponding 
ideal $\J$ satisfies $\Si=\V(\J)$.
\end{itemize}
\end{teo}
\begin{proof}
\begin{itemize}
\item[(i)] This is proposition \ref{p. dim 1} below.
\item[(ii)] This is corollary \ref{c. dim 0} below.
\end{itemize}
\end{proof}

\subsection{$\dim\Si=1$}\label{sub. dim 1}

In this section we show that under the assumption $\dim\Si=1$, we can always choose binomials $f_1,\ldots,f_r\in\IG$ satisfying 
the conditions of theorem \ref{t. ideal Nash} and such that $\Si=\V(\J)$. By corollary \ref{c. union orbits} our assumption implies
\[\Si=
\left\{
\begin{array}{rll}
&\overline{\Oo_1}\cup\overline{\Oo_2},	\\						
&\overline{\Oo_1},            		\\				
&\overline{\Oo_2}.										
\end{array}
\right.
\]
If $\Si=\overline{\Oo_1}\cup\overline{\Oo_2}$ then we obtain directly that $\V(\J)=\Si$ for any choice of $\J$ since 
$\Si\subseteq\VJ$ and $\VJ$ is also a union of closures of orbits. So we assume that $\Si=\overline{\Oo_1}=
\overline{\{(\zl,\zm,z_1,\ldots,z_n)\in X\}}$ (case $\Si=\overline{\Oo_2}$ follows similarly). 


\begin{lem}\label{l. M exists}
There is a $r\times r$ minor of $\Jc(f_1 \ldots f_s)$ equal to $x_1^{\alpha_1}\cdots x_l^{\alpha_l}$ for some 
$(\alpha_1,\ldots,\alpha_l)\in\N^l\setminus\{\zl\}$.
\end{lem} 
\begin{proof}
Suppose that every $r\times r$ minor of $\Jc(f_1 \ldots f_s)$ has some of the variables $y_j$ or $z_k$.
We want to contradict the fact $\Si=\overline{\Oo_1}$, so we need to 
find a point $(x_1,\ldots,x_l,\zm,\zn)\in\Si$, where $x_i\neq0$ for some $i$ (hence for all $i$ by corollary 
\ref{c. equations}). By corollary \ref{c. z-z}, $(\ul,\zm,\zn)\in X$ and the assumption on the minors implies
that $(\ul,\zm,\zn)\in\Si$. This contradiction implies that there must be 
a $r\times r$ minor of $\Jc(f_1,\ldots,f_s)$ of the form $x_1^{\alpha_1}\cdots x_l^{\alpha_l}$.
\end{proof}

\begin{pro}\label{p. dim 1}
Let $M$ be the submatrix of $\Jc(f_1\ldots f_s)$ formed by the $r$ rows having a minor of the form 
$x_1^{\alpha_1}\cdots x_l^{\alpha_l}$ for some $(\alpha_1,\ldots,\alpha_l)\in\N^l\setminus\{\zl\}$. 
Let $\J\subset\C[x,y,z]/\IG$ be the ideal generated by all $r\times r$ minors of $M$. Then $\VJ=\Si.$
\end{pro}
\begin{proof}
The existence of the submatrix $M$ is guaranteed by lemma \ref{l. M exists}. We already know that 
$\Si\subset\VJ$. Let $p=(x,y,z)\in\V(\J)$. In particular, $x_1^{\alpha_1}\cdots x_l^{\alpha_l}=0$. 
Thus $x_i=0$ for some $i$, hence $x_i=0$ for all $i\in\{1,\ldots,l\}$ according to corollary \ref{c. equations}. 
By the same corollary, we also have $y_j=0$ for all $j\in\{1,\ldots,m\}$. Therefore, $p=(0,0,z)\in\Si$.
\end{proof}

As we mentioned earlier, for a complete intersections $X$ we always have $\V(\J)=\Si$. The following two
examples show that the family of toric surfaces having singular locus of dimension 1 is another family
for which we always have $\V(\J)=\Si$. Indeed, using $\mathtt{SINGULAR}$ $\mathtt{4}$-$\mathtt{0}$-$\mathtt{2}$ 
(\cite{DGPS}), we verified that these example are not Cohen-Macaulay at the origin. This implies that both examples are 
not complete intersections.

\begin{exam}\label{e. two orbits}
Let $\Gamma_0=\{(2,0),(3,0),(2,6),(0,4),(0,5)\}\subset\Z^2$ and $\Gamma=\Z_{\geq0}\Gamma_0$.
Using \cite[Algorithm 4.5]{St} with the lex order, it can be shown that $I_{\Gamma}=\langle w^5-t^4,y^2t^6-z^3w^3,y^2w^2t^2-z^3,
xt^2-zw,xw^4-zt^2,xz^2-y^2w^3,x^2w^3-z^2,x^2zw-y^2t^2,x^3-y^2\rangle.$
The singular locus of $X=\V(\IG)$ is
$$\Si=\{(x,y,0,0,0)\in X\}\cup\{(0,0,0,w,t)\in X\}.$$
In particular, $\V(\J)=\Si$ for any choice of the ideal $\J$.
\end{exam}

\begin{exam}
Let $\Gamma_0=\{(2,0),(1,2),(0,3),(0,5)\}\subset\Z^2$. As before, $\IG=\langle z^5-w^3,xw^2-y^2z^2,xz^3-y^2w,x^2zw-y^4\rangle$.
Let $\J$ be the ideal generated by all $2\times2$ minors of the first two rows of $\Jc(\IG)$. Then $$\V(\J)=\{(x,0,0,0)\in X\}=\Si.$$
\end{exam}


\subsection{$\dim\Si=0$}

Recall the notation of remark \ref{r. rank}: $\IG=\langle f_1,\ldots,f_s \rangle$, 
$f_{\lambda}=x^{A_{\lambda}}y^{B_{\lambda}}z^{C_{\lambda}}-x^{A'_{\lambda}}y^{B'_{\lambda}}z^{C'_{\lambda}}$,
$\lambda=1,\ldots s$, and $f_1,\ldots,f_r$ are such that $\rk(\Jc(f_1,\ldots,f_r))=r$. In this section we show that 
under the assumption $\dim\Si=0$, we have $\Si\subsetneq\V(\J)$. 

\begin{rem}
Notice that $\dim\Si=0$ implies $\Si=\{\cero\}$.
\end{rem}

\begin{lem}\label{l. 0 sing pt}
Let $X$ be the toric surface defined by $\Gamma$ and let $J\subset\C[X]$ be a non-empty set of monomials. 
Then $\V(J)=\{\cero\}$ if and only if monomials of the form $x_1^{\alpha_1}\cdots x_l^{\alpha_l}$ and
$z_1^{\gamma_1}\cdots z_n^{\gamma_n}$ appear in $J$, for some $(\alpha_1,\ldots,\alpha_l)\in\N^l
\setminus\{\zl\}$ and $(\gamma_1,\ldots,\gamma_n)\in\N^n\setminus\{\zn\}$.
\end{lem}
\begin{proof}
If such monomials appear in $J$ then, by corollary \ref{c. equations}, $\V(J)=\{\cero\}$.
Now suppose that $\V(J)=\{\cero\}$. Suppose that every monomial in $J$ has some of the variables 
$x_i$ or $y_j$. By corollary \ref{c. z-z}, $(\zl,\zm,\un)\in\V(J)$. This is a contradiction. 
Therefore, a monomial of the form $z_1^{\gamma_1}\cdots z_n^{\gamma_n}$ must appear in $J$. 
Similarly, not all monomials in $J$ have the variables $y_j$ or $z_k$ so there must be a monomial of 
the form $x_1^{\alpha_1}\cdots x_l^{\alpha_l}$.
\end{proof}

For the following proposition we need to introduce some notation. For $K_1\subset\{1,\ldots,l\}$, we denote as $g^1_{K_1}:\{1,\ldots,l\}\rightarrow\{0,1\}$ 
its characteristic function. Similarly, $g^2_{K_2}$ and $g^3_{K_3}$ denote the characteristic functions on $\{1,\ldots,m\}$ and $\{1,\ldots,n\}$,
respectively. In addition, let $N=l+m+n$ and recall that $r=N-2$. Consider the following $(r\times N)$-matrix:
\[
M=\left( 
\begin{array}{ccc}
A_1 & B_1 & C_1             \\
\vdots & \vdots & \vdots   \\
A_r & B_r & C_r
\end{array} \right)
=\left( 
\begin{array}{ccccccccc}
A_{11}   & \cdots & A_{1l}  & B_{11} & \cdots & B_{1m} & C_{11} & \cdots & C_{1n}  \\
\vdots    &         & \vdots  &  \vdots  &          & \vdots& \vdots     &           & \vdots   \\
A_{r1}   & \cdots & A_{rl}  & B_{r1}  & \cdots & B_{rm} & C_{r1}   & \cdots & C_{rn}
\end{array} \right).
\]
\begin{pro}\label{p. dim 0}
Let $X\subset\C^{N}$ be the toric surface defined by $\Gamma$ and suppose that $X$ is not a 
hypersurface. Let $J$ be the set of monomials given by all $r\times r$ minors of $\Jc(f_1,\ldots,f_r)$. 
Then $J$ cannot contain, simultaneously, monomials of the form $x_1^{\alpha_1}\cdots x_l^{\alpha_l}$ 
and $z_1^{\gamma_1}\cdots z_n^{\gamma_n}$ for some $(\alpha_1,\ldots,\alpha_l)\in\N^l\setminus\{\zl\}$ and 
$(\gamma_1,\ldots,\gamma_n)\in\N^n\setminus\{\zn\}$. In particular, $\{\cero\}\subsetneq\V(J)=\V(\J)$.
\end{pro}
\begin{proof}
The last statement follows from lemmas \ref{l. minors} and \ref{l. 0 sing pt}.
To prove the proposition, we proceed by contradiction. Suppose that $J$ contains monomials 
$x^{\alpha}=x_1^{\alpha_1}\cdots x_l^{\alpha_l}$ and $z^{\gamma}=z_1^{\gamma_1}\cdots z_n^{\gamma_n}$. 
From lemma \ref{l. minors}, there exist $K_1,K'_1\subset\{1,\ldots,l\}$, $K_2,K'_2\subset\{1,\ldots,m\}$ 
and $K_3,K'_3\subset\{1,\ldots,n\}$ such that $|K_1\sqcup K_2\sqcup K_3|=|K'_1\sqcup K'_2\sqcup K'_3|=2$ and
\begin{align}
x^{\alpha}&\equiv \theta_1 x^{\textbf{A}-\ul}y^{\textbf{B}-\um}z^{\textbf{C}-\un}\cdot\prod_{i\in K_1}x_i\prod_{j\in K_2}y_j\prod_{k\in K_3}z_k,\mod\IG,\tag{*}\\
z^{\gamma}&\equiv \theta_2 x^{\textbf{A}-\ul}y^{\textbf{B}-\um}z^{\textbf{C}-\un}\cdot\prod_{i\in K'_1}x_i\prod_{j\in K'_2}y_j\prod_{k\in K'_3}z_k,\mod\IG,\tag{**}
\end{align}
for some $\theta_1,\theta_2\in\Z\setminus\{0\}$, where $\textbf{A}=\sum_{\lambda=1}^{r}A_{\lambda}$, $\textbf{B}=\sum_{\lambda=1}^{r}
B_{\lambda}$ and $\textbf{C}=\sum_{\lambda=1}^{r}C_{\lambda}$. Let $K=K_1\sqcup K_2\sqcup K_3$ and $K'=K'_1\sqcup K'_2\sqcup K'_3$. 
From lemma \ref{l. x-x}, $x^{\alpha}\not\equiv z^{\gamma}\mod\IG$. Therefore $K\neq K'$.
The congruence (*) gives place to the following set of equations (see lemma \ref{l. x-x}):
\begin{align}
\sum_{i=1}^l\alpha_i(a_i,b_i)&=\sum_{i=1}^l(\sum_{\lambda=1}^r A_{\lambda i}-1+g^1_{K_1}(i))(a_i,b_i),\label{e. eq 1.1}\\ 
0&=B_{1 j}+\cdots+B_{r j}-1+g^2_{K_2}(j), &&j\in\{1,\ldots,m\},\label{e. eq 1.2}\\
0&=C_{1 k}+\cdots+C_{r k}-1+g^3_{K_3}(k),&&k\in\{1,\ldots,n\}.\label{e. eq 1.3}
\end{align}
Similarly, the congruence (**) gives place to the following set of equations:
\begin{align}
0&=A_{1 i}+\cdots+A_{r i}-1+g^1_{K'_1}(i), &&i\in\{1,\ldots,l\},\label{e. eq 2.1}\\
0&=B_{1 j}+\cdots+B_{r j}-1+g^2_{K'_2}(j),&&j\in\{1,\ldots,m\},\label{e. eq 2.2}\\
\sum_{k=1}^n\gamma_k(e_k,f_k)&=\sum_{k=1}^n(\sum_{\lambda=1}^r C_{\lambda k}-1+g^3_{K'_3}(k))(e_k,f_k)).\label{e. eq 2.3}
\end{align}
By substituting equation (\ref{e. eq 2.1}) in (\ref{e. eq 1.1}) and equation (\ref{e. eq 1.3}) in (\ref{e. eq 2.3}) we obtain:
\begin{align}
\sum_{i=1}^l\alpha_i(a_i,b_i)&=\sum_{i=1}^l(\sum_{\lambda=1}^r g^1_{K_1}(i)-g^1_{K'_1}(i))(a_i,b_i),\notag\\ 
\sum_{k=1}^n\gamma_k(e_k,f_k)&=\sum_{k=1}^n(\sum_{\lambda=1}^r g^3_{K'_3}(k)-g^3_{K_3}(k))(e_k,f_k). \notag
\end{align}
To derive a contradiction from the previous equations we first observe that:
\begin{itemize}
\item[(i)] $K_1\neq\emptyset$ and $K'_3\neq\emptyset$.
\item[(ii)] $K'_1\subset K_1$ and $K_3\subset K'_3$.
\item[(iii)] $K_2=K'_2$.
\end{itemize}
Indeed, (i) follows from the fact that the right-hand part of (*) and (**) are monomials, and so $g^1_{K_1}(i)-g^1_{K'_1}(i)\geq0$,
for every $i\in\{1,\ldots,l\}$ and $g^3_{K'_3}(k)-g^3_{K_3}(k)\geq0$, for every $k\in\{1,\ldots,n\}$. To prove (ii) notice that $\alpha\in\N^l\setminus\{\zl\}$ 
and $\gamma\in\N^n\setminus\{\zn\}$. Finally, (iii) follows from combining equations (\ref{e. eq 1.2}) and (\ref{e. eq 2.2}).

From (i), we can consider the following cases:
\begin{itemize}
\item[(I)] $K_1=\{i_0\}$.
\item[(II)] $K_1=\{i_1,i_2\}$, where $i_1\neq i_2$. 
\end{itemize}
According to (ii), case (I) is in turn divided into the following two cases:
\begin{itemize}
\item[(I.1)] $K'_1=\{i_0\}$. Then $K_1=K'_1$. Since $K'_3\neq\emptyset$ it follows that $K_2=K'_2=\emptyset$ implying $K_3\neq\emptyset$.
Thus $K_3=K_3'$ so $K=K'$. This is a contradiction.
\item[(I.2)] $K'_1=\emptyset$. First notice that $K_1=\{i_0\}$ implies that $K_3=\{k_0\}$ or $K_3=\emptyset$. Suppose that $K_3=\{k_0\}$.
Then $K_2=\emptyset.$ Under these assumptions ($K'_1=\emptyset,K_3=\{k_0\},K_2=\emptyset$), equations (\ref{e. eq 1.2}), (\ref{e. eq 1.3}) and (\ref{e. eq 2.1}) 
imply that $C_{\lambda k_0}=0$ for every $\lambda\in\{1,\ldots,r\}$ and every other column of $M$ has an entry equal to 1 and all other entries in the column are 0.
Therefore, there are exactly $N-1$ entries equal to $1$ and the others are equal to $0$. According to remark \ref{r. minimal}, there must be at 
least two entries on every row of $M$ equal to $1$.  There are $r=N-2$ rows in $M$. Since $X$ is not a hypersurface, $N>3$. This is a contradiction.

If $K_3=\emptyset$ then $K_2=\{j_0\}$ and equations (\ref{e. eq 1.2}), (\ref{e. eq 1.3}) and (\ref{e. eq 2.1}) imply that $B_{\lambda j_0}=0$ for every $\lambda\in\{1,\ldots,r\}$ 
and some $j_0\in\{1,\ldots,m\}$, and every other column of $M$ has an entry equal to 1 and all other entries are 0. We arrive to the same contradiction.
\end{itemize}
Now we divide case (II) into the following three cases:
\begin{itemize}
\item[(II.1)] $K'_1=\{i_1,i_2\}$. In particular, $K'_3=\emptyset$. This contradicts (i).
\item[(II.2)] $K'_1=\{i_1\}$ or $K'_1=\{i_2\}$. First notice that $K_1=\{i_1,i_2\}$ implies that $K_2=K_3=\emptyset$. Now, exactly as in case (I.2), it follows that 
$A_{\lambda i_1}=0$ for every $\lambda\in\{1,\ldots,r\}$ and every other column of $M$ has an entry equal to 1 and all other entries are 0. 
We arrive to the same contradiction. Case $K'_1=\{i_2\}$ is analogous.
\item[(II.3)] $K'_1=\emptyset$. As in the previous item $K_2=K_3=\emptyset$. Then equations (\ref{e. eq 1.2}), (\ref{e. eq 1.3}) and (\ref{e. eq 2.1}) imply that every column of $M$ 
has an entry equal to 1 and all other entries are 0. Therefore, there are exactly $N$ entries equal to 1 in the matrix $M$ and all other entries are 0. 
In order not to contradict remark \ref{r. minimal} we must have $N=4$ and so $r=2$. 
\end{itemize}
Since $4=N=l+m+n$ and $|K_1|=2$, it follows that $l=2$, $m=0$, and $n=2$.
Since $\IG$ is a prime ideal, we only have two possible choices for 
$(A_{\lambda},C_{\lambda})\in\N^4$, $\lambda=1,2$:
\begin{itemize}
\item[(a)] $(A_{11},A_{12},C_{11},C_{12})=(1,0,1,0)$, $(A_{21},A_{22},C_{21},C_{22})=(0,1,0,1)$.
\item[(b)] $(A_{11},A_{12},C_{11},C_{12})=(1,0,0,1)$, $(A_{21},A_{22},C_{21},C_{22})=(0,1,1,0)$.
\end{itemize}
Let us consider case (a). The polynomials $f_1,f_2$ look like:
\begin{align}
f_1&=x_1z_1-x_2^{A'_{12}}z_2^{C'_{12}},\notag\\
f_2&=x_2z_2-x_1^{A'_{21}}z_1^{C'_{21}},\notag
\end{align}
where, necessarily, $A'_{12}\geq1$, $C'_{12}\geq1$, $A'_{21}\geq1$, and $C'_{21}\geq1$.
Assume that $A'_{12}\leq C'_{12}$. Since $f_1,f_2\in\IG$ we obtain the following congruences 
\begin{align}
x_1z_1&\equiv(x_2z_2)^{A'_{12}}z_2^{C'_{12}-A'_{12}}\mod\IG,\notag\\
x_2z_2&\equiv x_1^{A'_{21}}z_1^{C'_{21}}\mod\IG.\notag
\end{align}
In particular, $x_1z_1\equiv x_1^{A'_{12}A'_{21}} z_1^{A'_{12}C'_{21}}z_2^{C'_{12}-A'_{12}}\mod\IG$. 
Thus, 
$$x_1z_1(1-x_1^{A'_{12}A'_{21}-1} z_1^{A'_{12}C'_{21}-1}z_2^{C'_{12}-A'_{12}})\in\IG.$$
Since $\IG$ is a prime ideal which does not contain monomials, it follows that $1-x_1^{A'_{12}A'_{21}-1} z_1^{A'_{12}C'_{21}-1}z_2^{C'_{12}-A'_{12}}\in\IG$.
Now, $\cero\in X$ implies $A'_{12}A'_{21}=1$, $A'_{12}C'_{21}=1$, and $C'_{12}=A'_{12}$. We conclude that 
$A'_{12}=A'_{21}=1$ and $C'_{12}=C'_{21}=1$, i.e., $f_1=f_2$. This is a contradiction since
we are assuming that $\Jc(f_1\mbox{ } f_2)$  has rank 2.  If $A'_{12}\geq C'_{12}$ proceed similarly. 
Finally, a completely analogous argument shows that case (b) implies a similar contradiction.
\end{proof}

\begin{coro}\label{c. dim 0}
Let $X\subset\C^{N}$ be the toric surface defined by $\Gamma$ and suppose that $X$ is not a 
complete intersection. Let $\IG$ be the binomial ideal defining $X$. Then there is no choice of 
$r$ binomials in $\IG$ such that the corresponding ideal $\J$ satisfies $\V(\J)=\{\cero\}$.
\end{coro}
\begin{proof}
Since we are assuming that $\Si=\{\cero\}$, then being complete intersection is equivalent to being a hypersurface. 
This is because a surface which is a complete intersection and has isolated singularity is necessarily normal. But then we use the fact
that a normal toric surface is a complete intersection if and only if it is a hypersurface (see \cite{Ri}). The corollary then follows
directly from the proposition.
\end{proof}



\section{Some remarks on higher-dimensional toric varieties}

It is natural to ask whether the results obtained in previous sections have an analogue for higher-dimensional toric varieties. 
Unfortunately, the method used to prove our main theorem is unlikely to work for higher dimensions. Indeed, in our proof 
two key ingredients were required:

\begin{enumerate}
\item The existence of an specific set of defining equations for the toric surface (cf. corollary \ref{c. equations}).
\item A characterization of complete intersection for normal toric surfaces (see \cite{Ri}).
\end{enumerate}

As far as we know statements 1. and 2. do not have a higher-dimensional analogue. Despite these inconveniences, 
there are still some remarks that we can do regarding our problem in higher-dimensional toric varieties.

\begin{rem}
Let $X$ be a toric variety of any dimension such that $\Si$ is the union of the closures of all the orbits corresponding to 
maximal faces of the defining semigroup (for the correspondence between faces of a semigroup and orbits see \cite[Proposition 15]{GT}).
Then $\V(\J)=\Si$, for any ideal $\J$ defining the Nash blowup of $X$ (see the first paragraph of section \ref{sub. dim 1}).
\end{rem}

\begin{rem}
Let $X$ be a toric variety of dimension 3. Suppose that $\Si$ is the union of the closures of all but one of the orbits corresponding 
to maximal faces of the defining semigroup. Using a completely analogous argument to that of section \ref{sub. dim 1}, it can
be shown that we can always choose an ideal $\J$ defining the Nash blowup of $X$ satisfying $\V(\J)=\Si$. 
\end{rem}

We computed several examples of three-dimensional toric varieties in order to explore our question in that case. We were 
expecting to have a positive answer whenever the singular locus had maximal dimension as in the previous remark. Unfortunately,
we were not able to find any clear relation between the combinatorics of the faces defining the singular locus and
an answer to our question.



\section*{Acknowledgements}

The second author would like to thank Arturo Giles for explaining him his results appearing in \cite{Gi}. Those discussions 
motivated the question that we explored in this note.

\vspace{.5cm}
{\tiny \textsc {E. Ch\'avez Mart\'inez, Instituto de Matem\'aticas, UNAM, Unidad Cuernavaca.} Email: ecm\_2891@hotmail.com}\\
{\tiny \textsc {D. Duarte, Catedr\'atico CONACYT-UAZ.} Email: adduarte@matematicas.reduaz.mx}

\end{document}